\documentclass[12pt]{amsart}
\usepackage{amsfonts}
\usepackage{amsmath}

\newtheorem{proposition}{Proposition}
\newtheorem{lemma}{Lemma}
\newcommand{\legendre}{\overwithdelims()}

\usepackage{url}

\begin{document}
\title[The Landsberg-Schaar relation]{A proof of the Landsberg-Schaar relation by finite methods}
\author[Ben Moore]{Ben Moore}
\address{\hskip-\parindent
	School of Mathematical Sciences\\
	University of Adelaide\\ 
	SA 5005\\ 
	Australia}
\email{benjamin.moore@adelaide.edu.au}
\subjclass{11L05}
\begin{abstract}
	The Landsberg-Schaar relation is a classical identity between quadratic Gauss sums, normally used as a stepping stone to prove quadratic reciprocity. The Landsberg-Schaar relation itself is usually proved by carefully taking a limit in the functional equation for Jacobi's theta function. In this article we present a direct proof, avoiding any analysis.
\end{abstract}
\renewcommand{\subjclassname}{\textup{2010} Mathematics Subject Classification}

\maketitle

\section{Introduction}
\label{intro}
The aim of this article is to prove, using only techniques of elementary number theory, the Landsberg-Schaar relation for positive integral $a$ and $b$:
$$\frac{1}{\sqrt{a}}\sum_{n=0}^{a-1}{\exp{\left(\frac{2\pi in^2b}{a}\right)}}=\frac{1}{\sqrt{2b}}\exp{\left(\frac{\pi i}{4}\right)}\sum_{n=0}^{2b-1}{\exp{\left(-\frac{\pi i n^2a}{2b}\right)}}.$$
This relation was first discovered in 1850 by Mathias Schaar \cite{schaarproof}, who proved it using the Poisson summation formula, and proceeded to derive from it the law of quadratic reciprocity. In 1893 Georg Landsberg, apparently unaware of Schaar's work, rediscovered a slightly more general version of the relation \cite{landsbergproof}. Although Landsberg emphasises the role of modular transformations, his proof is closer in spirit to the modern one given in \cite{thetaimpliesrecip}, in which one takes a limit of the functional equation for Jacobi's theta function towards rational points on the real line.

A few remarks are in order concerning some closely related results involving techniques differing from those in the present article. Firstly, whilst this article was under review, the author noticed that in the article \cite{boylan-skoruppa}, the authors prove Hecke's generalisation of the Landsberg-Schaar identity over number fields. Their argument is elementary except for an appeal to Milgram's formula, which allows for the evaluation of exponential sums over non-degenerate integer-valued symmetric bilinear forms: the cited proof \cite[p. 127--131]{milnor-husemoller} uses Fourier analysis. When the number field is $\mathbb{Q}$, we recover the Landsberg-Schaar relation, and Milgram's formula in this instance is essentially Lemma \ref{Gauss} below. 

Secondly, the authors of \cite{boylan-skoruppa} suggest, in a parenthetical remark on the second page, that it does not seem possible to prove Hecke reciprocity by explicitly evaluating both sides. However, this does appear to work in the case of the Landsberg-Schaar relation. Indeed, in their book on Gauss and Jacobi sums \cite[Theorems 1.51, 1.52 and 1.54]{gaussbook}, Berndt, Evans and Williams give an elementary evaluation of $$\phi(a,b):=\sum_{n=0}^{a-1}{\exp{\left(\frac{2\pi in^2b}{a}\right)}}$$ for coprime positive integral $a$ and $b$. They deduce this result from Estermann's elementary evaluation \cite{estermann} of $$\phi(a,1):=\sum_{n=0}^{a-1}{\exp{\left(\frac{2\pi in^2}{a}\right)}}$$ for an odd positive integer $a$. One may then verify that $$\frac{1}{\sqrt{a}}\phi(a,b)=\frac{1}{2\sqrt{2b}}\exp{\left(\frac{\pi i}{4}\right)}\phi(4b,-a)$$ by evaluating both sides, and this equality is precisely the Landsberg-Schaar relation. In this argument, the hard work is contained in the evaluation of $\phi(a,b)$.

To emphasise the fact that the Landsberg-Schaar relation is an identity between Gauss sums, and to simplify the notation, we define, for $a$ and $b$ integers with $a>0$, $$\Phi(a,b)=\frac{1}{\sqrt{a}}\sum_{n=0}^{a-1}{\exp{\left(\frac{\pi i n^2b}{a}\right)}}.$$
Then the Landsberg-Schaar relation, for positive integral $a$ and $b$, takes the form
\begin{equation*}\label{LS}
\Phi(a,2b)=\sqrt{i}\Phi(2b,-a).
\end{equation*}

The starting point for our proof is the following evaluation of a quadratic Gauss sum, given by Gauss in 1811 \cite{gausseval}.

\begin{lemma}\label{Gauss}
	Let a be an integer, $a\geq1$. Then:
	\begin{equation*}
	\Phi(a,2)=
	\begin{cases}
	1+i & a=0 \bmod 4\\
	1 &   a=1 \bmod 4\\
	0 &   a=2 \bmod 4\\
	i &   a=3 \bmod 4.
	\end{cases}
	\end{equation*}
\end{lemma}

A proof of Lemma \ref{Gauss} avoiding analytical techniques may be given using linear algebra \cite{murtyeval}. Stronger results, which imply Lemma \ref{Gauss} (and Propositions \ref{phioddprimes} and \ref{phip=2} below), are also proved using elementary methods in \cite[Sections 1.3 and 1.5]{gaussbook}.

One may easily check that Lemma \ref{Gauss} is exactly the Landsberg-Schaar relation for $b=1$. Our aim is to prove the Landsberg-Schaar relation in general by induction on the number of distinct prime factors of $b$. The induction step follows from the next three results, and the bulk of this article is spent proving the third.

\begin{lemma}\label{multiplication}
	Let $a$, $b$ and $l$ be integers, $a$ positive and $(a,b)=1$. Then:
	\begin{equation*}
	\Phi(ab,l)=\Phi(a,bl)\Phi(b,al).
	\end{equation*}
\end{lemma}

The proof is not difficult, but is hard to find in this form: usually $l$ is assumed to be even, which simplifies matters considerably.

\begin{proof}
	As $s$ runs from $0$ to $b-1$ and $t$ runs from $0$ to $a-1$, $as+bt$ runs through a complete system of representatives for elements of $\mathbb{Z}/ab\mathbb{Z}$. So $$(as+bt)^2=g^2+2gkab+k^2a^2b^2$$
	for $k=0$ or $1$, $0\leq g <ab$. It follows that:
	$$\Phi(ab,l)=\frac{1}{\sqrt{ab}}\sum_{n=0}^{ab-1}{\exp{\left(\frac{\pi i n^2l}{ab}\right)}}=\epsilon\Phi(a,bl)\Phi(b,al),$$
	where
	\begin{equation*}
	\epsilon=\begin{cases}
	1 & a\text{ or } b \text{ even}\\
	{(-1)}^{S} & a\text{ and } b \text{ both odd,}
	\end{cases}
	\end{equation*}
	and $$S=\#\{(s,t) \mid as+bt>ab\}.$$
	
	The value of $S$ is $\frac{(a-1)(b-1)}{2}$ -- the problem of determining $S$ was set as a puzzle by Sylvester in \cite{sylvester} and solved by W. J. Curran Sharp in the same volume. The solution runs as follows: define $$P(x)=(1+x^b+x^{2b}+\dots+x^{ab})(1+x^a+x^{2a}+\dots+x^{ba}),$$ and note that $$P(x)=1+\dots+2x^{ab}+\dots+x^{2ab},$$ where the first dots comprise one term $x^g$ for each $g$ of the form $as+tb$ (we know that the coefficient of $x^g$ is $1$ since $a$ and $b$ are coprime).
	
	Since each factor of $P$ is a palindromic polynomial, so too is $P$, and it follows that the second dots comprise the same number of terms, all of coefficient 1. Therefore, $$(1+a)(1+b)=P(1)=4+2\#\{g<ab \mid g=as+tb\}.$$ Using the fact that $$\#\{g<ab \mid g=as+tb\}=(ab-1)-S,$$ the claim follows. So if $a$ and $b$ are both odd, then $S$ is even, and $\epsilon=1$ in this case too.
\end{proof}

The following result will not be needed until Section \ref{induction}.
\begin{lemma}\label{kfactor}
	Suppose $a$, $b$ and $k$ are nonzero integers, $a$ and $k$ are positive, and at least one of $a$ or $b$ is even. Then
	\begin{equation*}
	\Phi(ka,kb)=\sqrt{k}\Phi(a,b).
	\end{equation*}
\end{lemma}

\begin{proof}
	\begin{align*}
	\Phi(ka,kb)&=\frac{1}{\sqrt{ka}}\sum_{n=0}^{ka-1}{\exp{\left(\frac{\pi i n^2b}{a}\right)}}\\
	&=\frac{1}{\sqrt{k}}\sum_{m=0}^{k-1}{\frac{1}{\sqrt{a}}\sum_{n=0}^{a-1}{\exp{\left(\frac{\pi i (n+am)^2 b}{a}\right)}}}\\
	&=\frac{1}{\sqrt{k}}\sum_{m=0}^{k-1}{\exp{(i\pi abm^2)}}\Phi(a,b)\\
	&=\sqrt{k}\Phi(a,b).
	\end{align*}
\end{proof}

At this point, we only need one more result to prove the Landsberg-Schaar relation in Section \ref{induction}.
\begin{proposition}\label{reflection}
	Let $p$ be a prime and $l$ an integer with $(p,l)=1$. Then:	
	\begin{equation*}
	\Phi(p^k,2l)\Phi(p^k,-2l)=
	\begin{cases}
	1& p \text{ an odd prime, }k\geq 1\\
	2& p=2 \text{, }k\geq 3.
	\end{cases}
	\end{equation*}
\end{proposition}

The next two sections are devoted to proving Proposition \ref{reflection}, which is achieved by computing $\Phi(p^k,2l)$ directly. All the results of the next two sections are well-known in the literature, though apparently not all collected in one place. In particular, Proposition \ref{phioddprimes} and Proposition \ref{phip=2} are special cases of Gauss' evaluation of $\Phi(a,2)$, and may be found in \cite{gaussbook} as mentioned above. The proof of each proposition requires one to know the number of solutions to $x^2=a \bmod p^k$ for each $a$, which is the subject of the next section.

\subsection*{Acknowledgements} The author is extremely grateful to Mike Eastwood for his support and encouragement concerning this article, and most especially for his firm belief that an elementary proof of the Landsberg--Schaar relation should exist! The author would also like to thank Bruce Berndt for reading an earlier draft, Ram Murty for some encouraging remarks, David Roberts for tracking down Gauss' original evaluation of his eponymous sums, and the anonymous referee for suggesting valuable improvements to the article.

\section{Counting solutions to $x^2=a \bmod p^k$}
\label{counting}

The first result is reminiscent of Hensel's lemma, but is more direct.
\begin{lemma}\label{inductioncong}
	Let $p$ be a prime, not necessarily odd, and $j>i$.
	\begin{equation*}
	\#\{x \mid x^2=kp^i \bmod p^j\}=
	\begin{cases}
	p^{i/2}\#\{x|x^2=k \bmod p^{j-i}\} & i \text{ even}\\
	0 & \hspace{-0.5cm}i \text{ odd, } (k,p)=1.
	\end{cases}
	\end{equation*}
\end{lemma}

\begin{proof}
	To dispose of the case where $i$ is odd, note that $$x^2=kp^i+lp^j$$ implies $p^i$ divides $x$, so $p$ divides $k$. Now suppose that $i$ is even. Define 
	\begin{align*}
	A&=\{x \in \mathbb{Z}/p^j\mathbb{Z} \mid x^2=kp^i \bmod p^j\}\\
	B&=\{y \in \mathbb{Z}/p^{j-i}\mathbb{Z} \mid y^2=k \bmod p^{j-i}\}.
	\end{align*} The map $F:A\rightarrow B$ by $x\mapsto p^{i/2} x$ is surjective, so to prove that $|A|=p^{i/2}|B|$, we need only show that each fibre of $F$ has cardinality $p^{i/2}$.
	Since $F(x)=F(y)$ if and only if $x=y+tp^{j-i/2}$, the fibre of $F(x)$ contains nothing more than the elements $x_t=x+tp^{j-i/2}$ for $t=0,\dots,p^{i/2}-1$. But since $p^{i/2}$ divides $x$, 
	\begin{align*}
	{(x_{t})}^2&=x^2+2xtp^{j-i/2}+t^2p^{j+j-i/2} \bmod p^j\\
	&=x^2+2stp^{i/2}p^{j-i/2} \bmod p^j\\
	&=x^2 \bmod p^j.
	\end{align*} So the fibre of $F(x)$ is exactly the $x_t$. 
\end{proof}

We can now count the solutions to $x^2=0\bmod p^k$.
\begin{lemma}\label{zeros}
	\begin{equation*}
	\#\{x|x^2=0\bmod p^k\}=
	\begin{cases}
	p^{k/2} & k \text{ even}\\
	p^{(k-1)/2} & k \text{ odd}.
	\end{cases}
	\end{equation*}
\end{lemma}

\begin{proof}
	For $k$ even, put $j=k$, $i=k-2$ in Lemma \ref{inductioncong}. Then $$\#\{x|x^2=0\bmod p^k\}=p^{(k-2)/2}\#\{x|x^2=0\bmod p^2\},$$ and $$\{x|x^2=0\bmod p^2\}=\{0,p,2p,\dots,(p-1)p\}.$$
	For $k$ odd, put $j=k$, $i=k-1$ in Lemma \ref{inductioncong} to obtain $$\#\{x|x^2=0\bmod p^k\}=p^{(k-1)/2}\#\{x|x^2=0\bmod p\}=p^{(k-1)/2}.$$ 
\end{proof}

The next two results are standard: one may consult Hecke \cite[p. 47, Theorems 46a and 47]{hecke} or Dickson \cite[p. 13, Theorem 17]{dickson}.
\begin{lemma}\label{oddpcoprimek}
	Let $p$ be an odd prime, $j \geq 1$, $(k,p)=1$, and write ${k\legendre p}$ for the Legendre symbol. Then
	$$\#\{x \mid x^2=k \bmod p^j\}=1+{k\legendre p}.$$
\end{lemma}

\begin{lemma}\label{evenpcoprimek}
	For $p=2$ and $(k,2)=1$:
	\begin{equation*}
	\#\{x\mid x^2=k \bmod 4\}=
	\begin{cases}
	2 & k=1 \bmod 4\\
	0 & k=3 \bmod 4.
	\end{cases}
	\end{equation*}
	For $j\geq 3$:
	\begin{equation*}
	\#\{x\mid x^2=k \bmod 2^j\}=
	\begin{cases}
	4 & k=1 \bmod 8\\
	0 & \text{otherwise.}
	\end{cases}
	\end{equation*}	
\end{lemma}

Lemma \ref{inductioncong} and Lemma \ref{oddpcoprimek} taken together give us a complete picture for odd $p$ when $k \neq 0$, as follows.

\begin{lemma}\label{oddpcong}
	For $(k,p)=1$, $j>i$ and $p$ an odd prime:
	\begin{equation*}
	\#\{x\mid x^2=kp^i \bmod p^j\}=
	\begin{cases}
	p^{i/2}\left(1+{k\legendre p}\right) & i \text{ even}\\
	0 & i \text{ odd.}
	\end{cases}
	\end{equation*}
\end{lemma}

The analogue of Lemma \ref{oddpcong} for $p=2$ follows from Lemma \ref{inductioncong} and Lemma \ref{evenpcoprimek}. Since the exceptional cases $j=1$ and $j=2$ can be done by hand, we only need to consider $j-i \geq 3$.

\begin{lemma}\label{evenpcong}
	For $p=2$, $(k,2)=1$, $j\geq i+3$ and $i$ even:
	\begin{equation*}
	\#\{x\mid x^2=2^i k \bmod 2^j\}=
	\begin{cases}
	4p^{i/2} & k=1 \bmod 8\\
	0 & \text{otherwise.}
	\end{cases}
	\end{equation*}
	For $i$ odd (and all other hypotheses unchanged):
	\begin{equation*}
	\#\{x\mid x^2=2^i k \bmod 2^j\}=0.
	\end{equation*}	
\end{lemma}

\section{Evaluation of $\Phi(p^k,2l)$}
\label{eval}

This section is devoted to evaluating $\Phi(p^k,2l)$ for $p$ prime and $(l,p)=1$. We first evaluate $\Phi(p^k,2l)$ in the case that $p$ is an odd prime, then we proceed to the exceptional case of $p=2$. The idea of each proof is to expand $\Phi(p^k,2l)$ as a finite Fourier series. The coefficients have been calculated in Section \ref{counting}, and substituting in these expressions and simplifying yields the claimed results. In these calculations we implicitly make use of Lemma \ref{oddpcong}, Lemma \ref{evenpcong} and Lemma \ref{zeros}. We conclude this section with the proof of Proposition \ref{reflection}.

\begin{proposition}\label{phioddprimes}
	Let $p$ be an odd prime and $(l,p)=1$. Then:
	\begin{equation*}
	\Phi(p^k,2l)=
	\begin{cases}
	1 & k \text{ even, } k\geq 2 \\
	{l\legendre p}\Phi(p^k,2) & k \text{ odd.}
	\end{cases}
	\end{equation*}
\end{proposition}

\begin{proof}
	First we treat the case of $k$ even.
	\begin{align*}
	\Phi&(p^k,2l)=p^{-k/2}\sum_{n=0}^{p^k-1}{\#\{x|x^2=n\bmod p^k\}\exp{\left(\frac{2\pi i nl}{p^k}\right)}}\\
	&=p^{-k/2}\left[\sum_{n=0}^{p^k-1}{\left(1+{n\legendre p}\right)\exp{\left(\frac{2\pi inl}{p^k}\right)}}+(p^{k/2}-1)\right.\\
	&\left.-\sum_{n=1}^{p^{k-1}-1}{\exp{\left(\frac{2\pi inpl}{p^k}\right)}}+\sum_{i=2,4,\dots,k-2}^{}{\sum_{\substack{n=1\\(n,p)=p^i\\ n/p^i=m}}^{p^{k-i}-1}{\#\{x\mid x^2=mp^i \bmod p^k\}}}\right].
	\end{align*}
	We should explain each term in the last two lines: the first term gives the correct coefficients for $(n,p)=1$, the second term makes the correct contribution for $n=0$, the third term makes the $n$th coefficient $0$ for any nonzero $n$ divisible by $p$, and the final term restores the correct coefficient for these $n$. Note that the Legendre symbol ${n\legendre p}$ is defined to be $0$ if $n$ is divisible by $p$ -- this implies that ${\cdot \legendre p}$ is multiplicative.
	
	The inner sum in the last term can be simplified:
	\begin{align*}
	\sum_{\substack{n=1\\ (n,p)=p^i\\ n/p^i=m\\}}^{p^{k-i}-1}\hspace{-0.2cm}{\#\{x\mid x^2=mp^i \bmod p^k\}}&=p^{i/2}\left[\sum_{m=1}^{p^{k-i}-1}\hspace{-0.2cm}{\left(1+{m\legendre p}\right)\exp{\left(\frac{2\pi iml}{p^{k-i}}\right)}}\right.\\
	&\left.-\sum_{m=1}^{p^{k-i-1}-1}{\left(1+{mp\legendre p}\right)\exp{\left(\frac{2\pi iml}{p^{k-i-1}}\right)}}\right]
	\end{align*}
	\begin{align*}
	=p^{i/2}\sum_{m=1}^{p^{k-i}-1}{m\legendre p}\exp{\left(\frac{2\pi iml}{p^{k-i}}\right)}=p^{i/2}{l\legendre p}\sum_{m=0}^{p^{k-i}-1}{m\legendre p}\exp{\left(\frac{2\pi im}{p^{k-i}}\right)}.
	\end{align*}
	The final equality above follows from the facts that ${\cdot \legendre p}$ is multiplicative, and that $(l,p)=1$ implies that as $m$ runs from $0$ to $p^{k-i}-1$, so does $lm \bmod p^{k-i}$. But this last sum is zero, since $i\leq k-2$ implies $k-i>1$, and for $r>1$, $\sum_{m=0}^{p^{r}-1}{m\legendre p}\exp{\left(\frac{2\pi im}{p^{r}}\right)}=0$ as follows:
	\begin{align*}
	\sum_{m=0}^{p^{r}-1}{m\legendre p}\exp{\left(\frac{2\pi im}{p^{r}}\right)}&=\sum_{\alpha=0}^{p^{r-1}-1}{\sum_{n=\alpha p}^{(\alpha+1)p}{{n\legendre p}\exp{\left(\frac{2\pi in}{p^{r}}\right)}}}\\
	&=\sum_{\alpha=0}^{p^{r-1}-1}{\sum_{n=0}^{p-1}{{m+\alpha p \legendre p}\exp{\left(\frac{2\pi i(m+\alpha p)}{p^{r}}\right)}}}\\
	&=\sum_{\alpha=0}^{p^{r-1}-1}{\exp{(2\pi i\alpha p^{r-1})}}\sum_{n=0}^{p-1}{{m \legendre p}\exp{\left(\frac{2\pi im}{p^{r}}\right)}}\\
	&=0.
	\end{align*}
	Therefore, the last term in the expansion of $\Phi(p^k,2l)$ vanishes. When we expand the factor $(1+{n \legendre p})$ multiplying the first term, we find that the sum multiplied by $1$ is a geometric series, so it vanishes, and the sum multiplied ${n \legendre p}$ is zero by the calculation above. The $-1$ in the second term combines with the third term to give another geometric series, so we are left with $\Phi(p^k,2l)=1$, as promised.
	
	Now we treat odd $k$, and suppose for the moment that $k>1$. Then the coefficients are very similar, apart from the contributions for $n=0$ and \mbox{$n=mp^{k-1}$} with $(m,p)=0$. Specifically,
	\begin{align*}
	&\Phi(p^k,2l)=p^{-k/2}\left[\sum_{n=0}^{p^k-1}{\left(1+{n\legendre p}\right)\exp{\left(\frac{2\pi inl}{p^k}\right)}}+(p^{(k-1)/2}-1)\right.\\
	&\left.-\sum_{n=1}^{p^{(k-1)/2}-1}{\exp{\left(\frac{2\pi inpl}{p^k}\right)}}+\sum_{i=2,4,\dots,k-1}^{}{\sum_{\substack{n=1\\(n,p)=p^i\\n/p^i=m}}^{p^{k-i}-1}{\#\{x\mid x^2=mp^i \bmod p^k\}}}\right].
	\end{align*}
	The calculation above shows that each inner sum in the last term vanishes, except in the case $i=k-1$, in which the condition $k-i>1$ is no longer valid. So we consider this case separately:
	\begin{align*}
	\sum_{\substack{m=1\\m\neq\alpha p}}^{p-1}&{p^{(k-1)/2}(1+{m\legendre p})\exp{\left(\frac{2\pi i ml}{p}\right)}}\\
	&=p^{(k-1)/2}\left(-1+\sum_{m=1}^{p-1}{{m \legendre p}\exp{\left(\frac{2\pi iml}{p}\right)}}\right)\\
	&=p^{(k-1)/2}\left(-1+\sqrt{p}{l \legendre p}\Phi(p,2)\right).
	\end{align*}
	As with the case for $k$ even, the first term in the expansion of $\Phi(p^k,2l)$ vanishes, the $-1$ in the second term helps the third term vanish, and the last term only contributes $$p^{(k-1)/2}\left(-1+\sqrt{p}{l \legendre p}\Phi(p,2)\right),$$ so we are left with $$\Phi(p^k,2l)={l \legendre p}\Phi(p,2).$$
	
	If $k=1$, then it is clear that $\Phi(p,2l)={l \legendre p}\Phi(p,2)$ in this case too.  
\end{proof}

\begin{proposition}\label{phip=2}
	Suppose $k \geq 3$ and $(l,2)=1$. Then:
	\begin{equation*}
	\Phi(2^k,2l)=
	\begin{cases}
	\sqrt{2}\exp{\left(\frac{\pi i l}{4}\right)} & k \text{ odd}\\
	1+\exp{\left(\frac{\pi il}{2}\right)} & k \text{ even}
	\end{cases}
	\end{equation*}
\end{proposition}

\begin{proof}
	As before, for $k\geq 3$:
	\begin{align*}
	&\Phi(p^k,2l)=p^{-k/2}\sum_{n=0}^{p^k-1}{\#\{x \mid x^2=n \bmod p^k\}\exp{\left(\frac{2\pi inl}{p^k}\right)}}\\
	&=p^{-k/2}\left[\sum_{n=1 \bmod 8}^{p^k-1}{\#\{x \mid x^2=n \bmod p^k\}\exp{\left(\frac{2\pi inl}{p^k}\right)}}\right.\\
	&\left.+\sum_{i=1}^{k-1}{\sum_{(n,p)=p^i, n/p^i=m}^{~}\#\{x \mid x^2=mp^i \bmod p^k\}\exp{\left(\frac{2\pi imlp^{i}}{p^k}\right)}}+N\right],
	\end{align*}
	where $N=p^{k/2}$ if $k$ is even, and $N=p^{(k-1)/2}$ if $k$ is odd.
	
	Suppose $k \geq 4$ is even. Then:
	\begin{align*}
	\Phi(p^k,2l)&=p^{-k/2}\left[\sum_{n=1 \bmod 8}^{p^{k}-1}{4\exp{\left(\frac{2\pi inl}{p^k}\right)}}+p^{k/2}\right.\\
	&\left.+p^{k/2}\exp{\left(\frac{\pi i l}{2}\right)}+\sum_{i=2,4,\dots,k-4}^{~}{\sum_{n=1 \bmod 8}^{p^{k-i}-1}{4p^{i/2}\exp{\left(\frac{2\pi inlp^{i}}{p^k}\right)}}}\right]
	\end{align*}
	\vspace{-0.5cm}
	\begin{align*}
	=p^{-k/2}&\left[\sum_{\alpha=0}^{p^{k-3}-1}{4\exp{\left(\frac{2\pi i(1+p^3\alpha)l}{p^k}\right)}}+p^{k/2}+p^{k/2}\exp{\left(\frac{\pi il}{2}\right)}\right.\\
	&\left.+\sum_{i=2,4,\dots,k-4}^{~}{\sum_{\alpha=0}^{p^{k-i-3}-1}{4p^{i/2}\exp{\left(\frac{2\pi il(1+p^3\alpha)}{p^{k-i}}\right)}}}\right].
	\end{align*}
	Since $k\geq4$, and the term $i=k-2$ has been treated separately (and appears as the third term in the sum), we have $p^{k-i-3}-1>0$ for $i=2,4,\dots,k-4$, so the final sum is a geometric series and vanishes. Similarly, $k\geq4$ implies that the first sum vanishes too. Therefore $\Phi(2^k,2l)=1+\exp{\left(\frac{\pi i l}{2}\right)}$.\\
	Now suppose $k > 3$ is odd. This time, the term corresponding to $i=k-1$ appears separately as the third term in the brackets, and the term corresponding to $i=k-3$ appears as the fourth term.
	\begin{align*}
	\Phi(p^k,2l)&=p^{-k/2}\left[\sum_{\alpha=0}^{p^{k-3}-1}{4\exp{\left(\frac{2\pi i(1+p^3\alpha)l}{p^k}\right)}}+p^{(k-1)/2}\right.\\
	&\left.+p^{(k-1)/2}\exp{(\pi i l)}+4p^{(k-3)/2}\exp{\left(\frac{2\pi il}{8}\right)}\right.\\
	&\left.+\sum_{i=2,4,\dots,k-5}^{~}{\sum_{\alpha=0}^{p^{k-i-3}-1}{4p^{i/2}\exp{\left(\frac{2\pi il(1+p^3\alpha)}{p^{k-i}}\right)}}}\right].
	\end{align*}
	Then since $p^{k-i-3}-1>0$ for $i=2,4,\dots,k-5$, the final sum vanishes, as does the first sum, so we are left with:
	$$\Phi(2^k,2l)=p^{-1/2}\left(1+{(-1)}^{l}+2\exp{\left(\frac{\pi i l}{4}\right)}\right)=\sqrt{2}\exp{\left(\frac{\pi i l}{4}\right)}.$$
	Lastly, suppose $k=3$. Then compared to the case $k>3$ above, the extra term for $i=k-3$ is omitted, since this is the case $i=0$ which is already accounted for by the first term. For ease of comparison we write this expression out before explicitly setting $k=3$:
	\begin{align*}
	\Phi(p^k,2l)=p^{-k/2}&\left[\sum_{\alpha=0}^{p^{k-3}-1}{4\exp{\left(\frac{2\pi i(1+p^3\alpha)l}{p^k}\right)}}\right.\\
	&\left.+p^{(k-1)/2}+p^{(k-1)/2}\exp{(\pi i l)}\phantom{\sum_{\alpha=0}^{p^{k-3}-1}}\hspace{-30pt}\right].
	\end{align*}
	Now we set $k=3$ in the expression above, and simplify:
	\begin{align*}
	\Phi(p^k,2l)&=2^{-3/2}\left(4\exp{\left(\frac{\pi i l}{4}\right)}+2+2(-1)^{l}\right)=\sqrt{2}\exp{\left(\frac{\pi i l}{4}\right)}. 
	\end{align*}  
\end{proof}

Finally, we can prove Proposition \ref{reflection}:
Let $p$ be an odd prime, $k\geq1$, and $(p,l)=1$:
\begin{align*}
\Phi(p^k,2l)\Phi(p^k,-2l)&={l\legendre p}{-l\legendre p}{\left(\Phi(p^k,2)\right)}^2 \text{ (by Proposition \ref{phioddprimes})}\\
&={-1\legendre p}{{l\legendre p}}^2{-1\legendre p} \text{ (by Lemma \ref{Gauss})}\\
&=1.
\end{align*}
Let $p=2$ and suppose $k\geq3$, $(2,l)=1$. By Proposition \ref{phip=2}:
\begin{align*}
\Phi(2^k,2l)\Phi(2^k,-2l)&=
\begin{cases}
\left(1+\exp{\left(\frac{\pi il}{2}\right)}\right)\left(1+\exp{\left(-\frac{\pi il}{2}\right)}\right) & k \text{ even}\\
2\exp{\left(\frac{\pi i l}{4}\right)}\exp{\left(-\frac{\pi i l}{4}\right)} & k \text{ odd}
\end{cases}\\
&=2.
\end{align*}

\section{Induction}
\label{induction}

By Lemma \ref{Gauss}, the Landsberg-Schaar relation holds for $b=1$. We proceed by induction on the number of distinct prime factors of $b$. We assume that $\Phi(a,2b)=\sqrt{i}\Phi(2b,-a)$ for all $b$ with less than $n$ prime factors, and prove, using Proposition \ref{reflection}, that $\Phi(a,2bp^k)=\sqrt{i}\Phi(2bp^k,-a)$ for all primes $p$. We may assume that $(b,p)=1$, and also $(a,p)=1$ by Lemma \ref{kfactor}. As usual, the case for $p$ an odd prime is treated first.
\begin{align*}
\Phi(a,2bp^k)&=\frac{\Phi(p^ka,2b)}{\Phi(p^k,2ab)} &\text{(by Lemma \ref{multiplication})}\\
&=\frac{\sqrt{i}\Phi(2b,-p^ka)}{\Phi(p^k,2ab)} &\\
&=\frac{\sqrt{i}\Phi(2bp^k,-a)}{\Phi(p^k,2ab)\Phi(p^k,-2ab)} &\text{(by Lemma \ref{multiplication})}\\
&=\sqrt{i}\Phi(2bp^k,-a). &\text{(by Proposition \ref{reflection})}
\end{align*}

Now for $p=2$:
\begin{align*}
\Phi(a,2b.2^k)&=\Phi(a,2^{k+1}b)=\frac{\Phi(2^{k+1}a,b)}{\Phi(2^{k+1},ab)} &\text{(by Lemma \ref{multiplication})}\\
&=\frac{\Phi(2^{k+2}a,2b)}{\Phi(2^{k+2},2ab)} &\text{(by Lemma \ref{kfactor})}\\
&=\frac{\sqrt{i}\Phi(2b,-2^{k+2}a)}{\Phi(2^{k+2},2ab)} &\\
&=\frac{\sqrt{2}\sqrt{i}\Phi(b,-2^{k+1}a)}{\Phi(2^{k+2},2ab)} &\text{(by Lemma \ref{kfactor})}\\
&=\frac{\sqrt{2}\sqrt{i}\Phi(2^{k+1}b,-a)}{\Phi(2^{k+2},2ab)\Phi(2^{k+1},-ab)} &\text{(by Lemma \ref{multiplication})}\\
&=\frac{\sqrt{i}\Phi(2b.2^k,-a)}{\frac{1}{2}\Phi(2^{k+2},2ab)\Phi(2^{k+2},-2ab)} &\text{(by Lemma \ref{kfactor})}\\
&=\sqrt{i}\Phi(2b.2^k,-a). &\text{(by Proposition \ref{reflection})}
\end{align*}

\end{document}